%% file: M5.tex
\documentclass[12pt,reqno]{amsart}
\usepackage{amssymb,amscd}
\usepackage[alphabetic]{amsrefs}
\usepackage[all]{xy}
\usepackage[headings]{fullpage}
\usepackage{libertine}
\usepackage{tikz}

\include{formatting}
\numberwithin{equation}{section}

\include{macros}
\renewcommand{\and}{\quad\text{and}\quad}

\def\J{\mathcal{J}}
\def\HJ{H^J}
\def\HlJ{H^{\le J}}

\def\prH{{}'H^J}
\def\pH{{}^pH^J}
\newcommand{\Rs}{\mathsf{R}}
\newcommand{\Cs}{\mathsf{C}}

\subjclass{Primary 11G20, 14D10, 14H30;
Secondary 11C20, 14G17, 15B33}

\keywords{Curve, Artin-Schreier cover, Cartier operator, $a$-number,
  Jacobian, $p$-torsion, Ekedahl--Oort type}

\begin{document}

\begin{abstract}
  Let $k$ be a perfect field of characteristic $p>0$, and let $d$ be a
  positive integer not divisible by $p$.  We define a non-empty
  Zariski open subset $U$ of the space of polynomials of degree $d$,
  and for $f(x)\in U(k)$, we compute the $a$-number of the curve
  defined by $y^p-y=f(x)$.  This $a$-number realizes a lower bound
  given by Booher and Cais, so the latter is tight.  Our result also
  implies that the bound of Booher and Cais for minimal $a$-numbers of
  Artin-Schreier covers of ordinary curves is tight.
\end{abstract}

\title[Minimal $a$-numbers]{Minimal $a$-numbers of Artin--Schreier
  covers of ordinary curves} 

\author{Bryden Cais}
\address{Department of Mathematics \\ University of Arizona
 \\ Tucson, AZ~~85721 USA}
\email{cais@arizona.edu}

\author{Douglas Ulmer}
\address{Department of Mathematics \\ University of Arizona
 \\ Tucson, AZ~~85721 USA}
\email{ulmer@arizona.edu}

%\date{\today}

\maketitle

\section{Introduction}
Our aim in this paper is to give a construction of curves with minimal
$a$-number in the context of Artin--Schreier covers.

\subsection{Background}
Let $k$ be a perfect field of characteristic $p$ with algebraic
closure $\kbar$, and let $Z$ be a smooth, irreducible, projective
curve of genus $g_Z$ over $k$.  There is an additive, $p^{-1}$-linear
map called the Cartier operator
\[\CC: H^0(Z,\Omega^1_Z)\to H^0(Z,\Omega^1_Z)\]
defined as follows.  Choose $x\in k(Z)\setminus k(Z)^p$.  Then
$\{1,x,\dots,x^{p-1}\}$ is a basis of $k(Z)$ as a vector space over
the subfield $k(Z)^p$.  Given $\omega\in H^0(Z,\Omega^1_Z)$, write
\[\omega=\left(f_0+f_1x+\cdots+f_{p-1}x^{p-1}\right)\,dx,\]
with $f_0,\dots,f_{p-1}\in k(Z)^p$.  Then $\CC$ is defined by
\[\CC(\omega)=f_{p-1}^{1/p}\,dx.\]
Clearly $\CC$ is additive in $\omega$.  The phrase ``$p^{-1}$-linear''
means that $\CC(f\omega)=f^{1/p}\CC(\omega)$ for $f\in k(Z)^p$.  See
\cite{Cartier57} for the original definition and \cite{AchterHowe19}
for a careful account of representing semi-linear operators by
matrices in this and related contexts.

The \emph{$a$-number} of $Z$ is defined as
\[a_Z:=\dim_k\ker\left(\CC: H^0(Z,\Omega^1_Z)
    \to H^0(Z,\Omega^1_Z)\right).\]

The terminology comes from the fact that $a_Z$ is the largest integer
$a$ such that there is an injection of $k$-group schemes
$\left(\alpha_p\right)^a\into J_Z[p]$,
where $J_Z[p]$ is the $p$-torsion group scheme of the Jacobian of $Z$.
The integer $a_Z$ is also a measure of the complexity of $J_Z[p]$ in
that it is the minimal number of generators of the local-local part of
the Dieudonn\'e module of $J_Z[p]$ over the Dieudonn\'e ring.  Clearly
$0\le a_Z\le g_Z$, and it is known that $a_Z=0$ if and only if $Z$
satisfies $J_Z[p](\kbar)=(\Z/p\Z)^{g_Z}$, and $a_Z=g_Z$ if and only if
$J_Z$ is isomorphic over $\kbar$ to the product of $g_Z$ copies of a
supersingular elliptic curve.
% Need \kbar: See Oort, Math A, 1975
We omit details as neither of these connections will be needed in this
article, but see \cite{Pries08} for a quick introduction to these and
related results.

\subsection{Results}
Let $X=\P^1$ over $k$ with coordinate $x$.  Let $d$ be a positive
integer not divisible by $p$, and let $A\subseteq\A^{d+1}$ be the open
subset of tuples $(a_0,\dots,a_d)$ with $a_d\neq0$.  For
$a=(a_0,a_1,\dots,a_d)\in A(k)$ let
\[f_a(x)=a_0+a_1x+\cdots+a_dx^d\in k[x].\]
Let $\pi:Y_a\to X$ be the Artin--Schreier cover of $X$ defined by
\[Y=Y_a:\qquad y^p-y=f_a(x).\]
It is well known (see, e.g., \cite[Prop.~6.4.1]{StichtenothAFFC}) that
the genus of $Y$ is $(p-1)(d-1)/2$. 

Booher and Cais gave lower and upper bounds on the $a$-number of $Y$
in \cite{BooherCais20}.  These upper and lower bounds coincide when
$p=2$, so both are sharp.  Thus we may assume from now on that $p>2$.
The Booher--Cais lower bound was simplified by Booher, Pries and
several collaborators in \cite{BooherPriesetal22} (initiated at PROMYS
2019), and they proved that the lower bound is optimal for $p=3$ and
$p=5$.  In \cite{Shi26}, Shi gives further examples of Artin--Schreier
covers with minimal $a$-number for every $p$. 

Our main result shows that the Booher--Cais bound is optimal for every
$p$ and every pole order $d$.  We write $\lfloor x \rfloor$ for the
largest integer $\le x$, and for $d$ prime to $p$ we define an integer
\[L(d):=\sum_{\ell=1}^{(p-1)/2}
    \left\lfloor\frac{(p-\ell)d}{p}\right\rfloor
    -\left\lfloor\frac{((p^2+1-2\ell p)d}{2p^2}\right\rfloor.\]

\begin{thm}\label{thm:main}
  There is a non-empty Zariski open subset $U$ of $A$ such that if $a\in
  U(k)$, then the $a$-number of $Y$ is $L(d)$.
\end{thm}

(This is, up to a simple change of variables, the bound
$L_{(p+1)/2}(d)$ given in \cite{BooherPriesetal22}.)

Our result also has consequences for Artin--Schreier covers of general
ordinary curves $X$, i.e., curves with $a_X=0$.  Let $X$ be a (smooth,
irreducible, projective) ordinary curve over $k$, and
let $\pi:Y\to X$ be a ramified Artin--Schreier cover of $X$ defined by
\[Y: y^p-y=f,\quad\text{for $f\in k(X)$}.\]
If $P_1,\dots,P_B$ are the branch points of $\pi$, let $d_1,\dots,d_B$
be the ramification invariants of $\pi$, i.e., let $d_b$ be the
unique break in the lower ramification filtration of $\gal(Y/X)$ at
$P_b$.
The integers $d_b$ are necessarily prime to $p$.

\begin{cor}
  Suppose that $k$ is algebraically closed and $X$ is ordinary, and
  choose an integer $B\ge 1$, points $P_1,\dots,P_B\in X(k)$, and positive
  integers $d_1,\dots,d_B$ all relatively prime to $p$.  Then there
  exists an Artin--Schreier cover $\pi:Y\to X$ ramified precisely at
  $P_1,\dots,P_B$ with ramification invariants $d_1,\dots,d_B$ such
  that $a_Y$ is minimal among all covers ramified at $B$ points with
  invariants $d_1,\dots,d_B$, namely
  \[a_Y= \sum_{b=1}^B L(d_b).\]
\end{cor}

This follows from the formal patching result of
\cite[Thm.~1.2]{BooherPriesetal22} (itself based on
\cite[\S2.3]{Booherpries00}) and our 
Theorem~\ref{thm:main}. 

\subsection{Strategy}
The structure of the formula in Theorem~\ref{thm:main} reflects our
strategy of proof, namely the use of a certain filtration on
$H=H^0(Y,\Omega^1_Y)$.  In Section~\ref{s:fils}, we define subspaces
\[0\subseteq H^{\le1}\subseteq H^{\le2}\subseteq\cdots\subseteq H^{\le
    p-1}=H,\]
and in Theorem~\ref{thm:ref1} we compute the dimension of the
kernel of $\CC$ restricted to each $H^{\le J}$ (with $a\in U(k)$ for a certain
non-empty open $U\subseteq A\subseteq\A^{d+1}$).  The proof of
Theorem~\ref{thm:ref1} is reduced to Theorem~\ref{thm:ref2} which
asserts the surjectivity or injectivity of certain maps between
subquotients of $H$.  To establish Theorem~\ref{thm:ref2}, we regard
the coefficients $a_0,\dots,a_d$ as variables and prove that a certain
determinant is a non-zero polynomial in these variables, and so the
set where it vanishes is a proper Zariski closed subset of the space
of coefficients.

In the end, we find more generally that for $J\le (p+1)/2$, and
generic $a$,
  \begin{equation}\label{eq:ker-fil}
\dim_k\ker\left(\CC:H^{\le J}\to H\right)=
    \sum_{\ell=1}^J
    \left\lfloor\frac{(p-\ell)d}{p}\right\rfloor
    -\left\lfloor\frac{(p(J+1-\ell)-J)d}{p^2}\right\rfloor,    
  \end{equation}
  and that for $J\ge(p-1)/2$,
  \[\ker\CC|_{H^{\le J}}=\ker\CC|_{H^{\le(p-1)/2}}.\]
The main theorem is an immediate consequence.  \emph{A posteriori}, we
see that the right-hand side of \eqref{eq:ker-fil} is the quantity
$L_J(d)$ studied in \cite{BooherPriesetal22}.

\section{Filtrations and refined statements}\label{s:fils}
Recall that $k$ is a perfect field of characteristic $p$.  Let
$X=\P^1$ over $k$.
Fix an $a\in A(k)$ and let $f=f_a$ and $Y=Y_a$
as in the introduction.  Let $Q_\infty$ be the unique point of $Y$ over
$\infty\in X$.

\begin{defn}\label{def:omega}
For  $0<j<p$ and $0<i<(p-j)d/p$, 
define differentials on $Y$ by
\[\omega_{i,j}=y^{j-1}x^{i-1}\,dx.\]
\end{defn}

One computes that
\[\ord_{Q_\infty}\omega_{ij}=pd-ip-jd-1\ge0,\]
so the $\omega_{i,j}$ are regular on $Y$.

The values of $\ord_{Q_\infty}\omega_{i,j}$ are easily seen to be
distinct modulo $pd$, so they are distinct integers which implies that
the $\omega_{i,j}$ are linearly independent.  Since there are
$(p-1)(d-1)/2=g_Y$ of them, the $\omega_{i,j}$ form a
basis of $H^0(Y,\Omega^1_Y)$.

We now define certain subspaces of $H:=H^0(Y,\Omega^1_Y)$ which will
be useful in thinking about the domain of $\CC$.

\begin{defn}\label{def:HJ}
For $0<J<p$, define
\begin{align*}
\HlJ&=\left\<\left.\omega_{i,j}\right| 0<j\le J,
        0<i<(p-j)d/p\right\>\\
\HJ&=\left\<\left.\omega_{i,j}\right| j=J,
        0<i<(p-j)d/p\right\>\\
\pH&=\left\<\left.\omega_{i,j}\right| j=J,
        0<i<(p-j)d/p, p|i\right\>\\
\prH&=\left\<\left.\omega_{i,j}\right| j=J,
        0<i<(p-j)d/p,p\nodiv i\right\>.
\end{align*}
(Here $\<\cdots\>$ denotes the $k$-span of $\cdots$.)  
\end{defn}

It is clear that
\[\HlJ=\bigoplus_{\J\le J}H^{\J}\and \HJ=\pH\bigoplus\prH.\]

The following subspaces will be useful in considering the range of
$\CC$.

\begin{defn}\label{def:H(l,e)}
  For $0<\ell<p$ and $\ell\le e<p$, let
\[  H(\ell,e):=\left\<\omega_{i,j}\left|j=\ell,
      0<i<\left\lfloor\frac{(p(e+1-\ell)-e)d}{p^2}\right\rfloor
    \right.\right\>\]
and we set  $H(\ell,e)=0$ when $e<\ell$.   (Here $\lfloor x\rfloor$
denotes the largest integer $\le x$.)
\end{defn}

The letter $\ell$ stands for ``level'' and $e$ for ``extent''.  The
definition is arranged to make Proposition~\ref{prop:containments}
true.

To help digest these definitions, consider Figure~\ref{fig:I}
where the dots denote differentials $\omega_{i,j}$ for $i$ and $j$ as
indicated.  Then $\HJ$ is the span of the row at level $J$, $\HlJ$
denotes the span of all of the rows up to level $J$, and $H(\ell,e)$
is the span of the part of the row at level $j=\ell$ to the left of a
point determined by $e$.  This point moves to the right as $e$
increases.

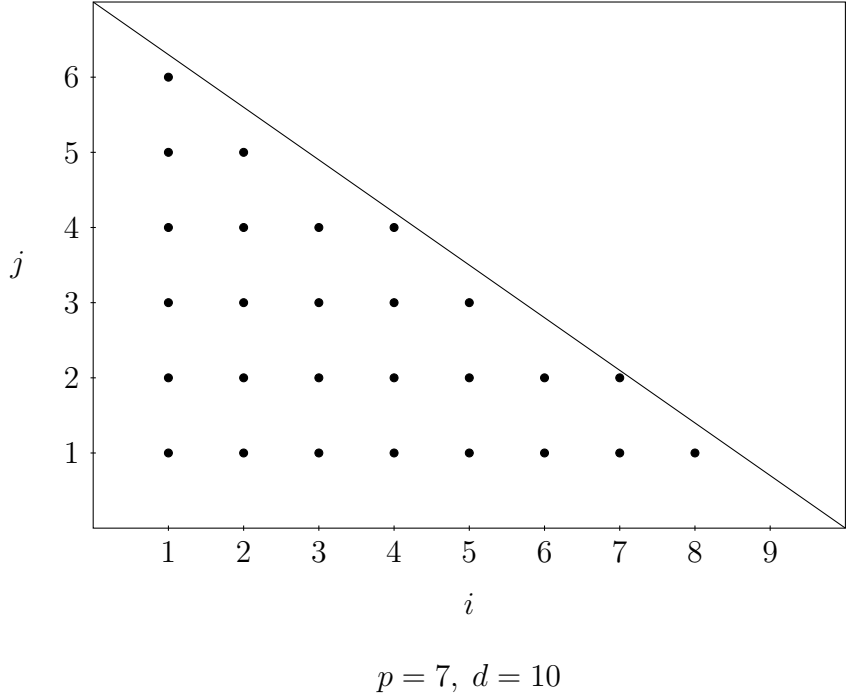
\begin{figure}
\begin{center}
\begin{tikzpicture}[scale=1]
  \draw[step=1cm] (0,0) rectangle (10,7);
  \foreach \x in {1,2,3,4,5,6,7,8,9}
   \draw (\x cm,1pt) -- (\x cm,-1pt) node[anchor=north] {$\x$};
  \foreach \y in {1,2,3,4,5,6}
   \draw (1pt,\y cm) -- (-1pt,\y cm) node[anchor=east] {$\y$};
  \node (i) at (5,-1) {$i$};
  \node (j) at (-1,3.5) {$j$};
  \node (l) at (5,-2) {$p=7,\ d=10$};
  \foreach \x in {1,2,3,4,5,6,7,8}
  \filldraw[black] (\x,1) circle (1.5pt);
  \foreach \x in {1,2,3,4,5,6,7}
  \filldraw[black] (\x,2) circle (1.5pt);
  \foreach \x in {1,2,3,4,5}
  \filldraw[black] (\x,3) circle (1.5pt);
  \foreach \x in {1,2,3,4}
  \filldraw[black] (\x,4) circle (1.5pt);
  \foreach \x in {1,2}
  \filldraw[black] (\x,5) circle (1.5pt);
  \foreach \x in {1}
  \filldraw[black] (\x,6) circle (1.5pt);
  \draw[step=1cm] (0,7) -- (10,0);
\end{tikzpicture}
\end{center}
\caption{An index set for differentials}\label{fig:I}
\end{figure}
\bigskip

\begin{rems}\mbox{}
  \begin{enumerate}
  \item The subspace $\HlJ$ can be defined intrinsically using the
    action of $\gal(Y/X)$ on $H$.  As far as we know, the subspaces
    $H(\ell,e)$ do not have an intrinsic definition (although they are
    in some sense characterized by Theorem~\ref{thm:ref1}).
  \item In a forthcoming %future\bry{forthcoming?}
    paper (generalizing \cite{CaisUlmer25} to ramified covers), we use
    a completed diagram with dots for all $0<j<p$, $0<i<d$ to index a
    convenient basis for $H^1_{dR}(Y)$ and use it to study the
    $p$-torsion group scheme $J_Y[p]$.
  \end{enumerate}
\end{rems}

As we will see below, the following result refines
Theorem~\ref{thm:main}.

\begin{thm}\label{thm:ref1} There is a non-empty Zariski open subset $U\subseteq
  A$ such that if $a\in U(k)$, and $Y=Y_a$, we have:
  \begin{enumerate}
  \item For $0<J\le(p+1)/2$, the restriction of $\CC$ to
    $H^{\le J}$ induces a surjection    
    \[H^{\le J}\labeledlongto{\CC}
     \bigoplus_{\ell=1}^{J}
      H(\ell,J).\]
  \item We have
    \[\ker\left(H\labeledlongto{\CC} H\right)\subseteq H^{\le (p-1)/2}.\]
  \end{enumerate}
\end{thm}

\begin{proof}[Proof that Theorem~\ref{thm:ref1} implies
  Theorem~\ref{thm:main}]
  Part (2) of Theorem~\ref{thm:ref1} says that the kernel of $\CC$ on
  $H$ is contained in $H^{\le(p-1)/2}$.  Using
  Definitions~\ref{def:omega} and \ref{def:HJ}, we have
  \[\dim \left(H^{\le(p-1)/2}\right)=
    \sum_{\ell=1}^{(p-1)/2}\left\lfloor\frac{(p-\ell)d}{p}\right\rfloor.\]
  The case $J=(p-1)/2$ of part (1) of Theorem~\ref{thm:ref1} implies
  that the dimension of the image (i.e., the codimension of the kernel)
  of $\CC$ on $H^{\le(p-1)/2}$ is
  \[\dim \left(\bigoplus_{\ell=1}^{(p-1)/2} H(\ell,(p-1)/2)\right).\]
  The summands on the right have dimensions
  $\left\lfloor\frac{(p^2+1-2\ell p)d}{2p^2}\right\rfloor$, so the
  dimension of the kernel of $\CC$ is
\[\sum_{\ell=1}^{(p-1)/2}
    \left\lfloor\frac{(p-\ell)d}{p}\right\rfloor
    -\left\lfloor\frac{(p^2+1-2\ell p)d}{2p^2}\right\rfloor,\]
as desired.
\end{proof}

We continue by refining the statement of Theorem~\ref{thm:ref1}.

\begin{prop}\label{prop:containments}
  For $0<J<p$ and  $a\in A(k)$, we have
  \begin{enumerate}
  \item
\[    \CC\left(\HJ\right)\subseteq\CC\left(\HlJ\right)\subseteq
    \bigoplus_{\ell=1}^J H(\ell,J).\]
  \item
\[    \CC\left(\prH\right)\subseteq
    \bigoplus_{\ell=1}^{J-1} H(\ell,J).\]
\item The map $\pH\labeledto{\CC}\CC(\pH)\into\CC(\HJ)$ composed with
  the inclusions in (1) and the projection onto the factor $\ell=J$ 
  induces an isomorphism
  \[\pH\isoto H(J,J).\]
\end{enumerate}
\end{prop}

\begin{proof}
  Let $\omega_{i,j}$ be a basis element of $\HlJ$, so $j\le J$ and
  $0<i<(p-j)d/p$.   Then we have
  \begin{align}
    \CC(\omega_{i,j})&=\CC\left(y^{j-1}x^{i-1}\,dx\right)\notag\\
                     &=\CC\left((y^p-f)^{j-1}x^{i-1}\,dx\right)\notag\\
                     &=\sum_{\ell=1}^j\binom{j-1}{\ell-1}(-1)^{j-\ell}
                       \CC\left(y^{p(\ell-1)}f^{j-\ell}x^{i-1}\,dx\right)\notag\\
                     &=\sum_{\ell=1}^j\binom{j-1}{\ell-1}(-1)^{j-\ell}
                       y^{\ell-1}\CC\left(f^{j-\ell}x^{i-1}\,dx\right).\label{eq:CC}
  \end{align}
  The $\ell$-th term of the sum
  lies in $H^{\ell}$ and since the degree of $f^{j-\ell}x^{i}$ is at
  most $(j-\ell)d+(p-j)d/p$, we have that
  $\CC\left(f^{j-\ell}x^{i-1}\,dx\right)$ is $g(x)x^{-1}\,dx$ where
  $g$ has degree at most
\[\frac{(j-\ell)d+(p-j)d/p}{p}=\frac{(p(j+1-\ell)-j)d}{p^2},\]
so the $\ell$-th term of the sum
lies in $H(\ell,j)\subseteq H(\ell,J)$.  This completes the proof of
part (1).

For part (2), note that the $\ell=j$ term of the sum
in equation~\eqref{eq:CC} is $y^{j-1}\CC\left(x^{i-1}\,dx\right)$, and
this is zero if $p\nodiv i$.  Thus if $\omega_{i,j}\in \prH$, the part
of $\CC(\omega_{i,j})$ in the $J$-th factor of the sum in part (1)
vanishes.  This establishes part (2).

  Part (3) follows by similar reasoning:  if $p|i$,
  \[\CC\left(y^{j-1}x^{i-1}\,dx\right)=y^{j-1}x^{i/p-1}\,dx+\cdots\]
  where the omitted terms involve lower powers of $y$.  Thus the map
  under consideration sends $\omega_{i,j}$ to $\omega_{i/p,j}$.  Let
  \[I=\left\lfloor\frac{(p-J)d}{p^2}\right\rfloor
    =\left\lfloor\frac1p\left\lfloor\frac{(p-J)d}{p}\right\rfloor\right\rfloor.\]
  Since the differentials $y^{j-1}x^{i-1}\,dx$
  for $j=J$ and $0<i\le I$ form a basis of $H(J,J)$, it
  follows that as  $\omega_{i,j}$ runs through the basis 
  \[\{\omega_{p,j},\omega_{2p,j},\dots,\omega_{pI,j}\}\]
  of $\pH$, the projection into $H(J,J)$ of the differentials
  $\CC(\omega_{i,j})$ runs through a basis of $H(J,J)$. 
  This establishes the claim in part (3).
\end{proof}

\begin{defn}\label{def:Phi}
For $1< J<p$, define
a map $\Phi_J$ as the composition
\[
\Phi_J: \quad    \prH\labeledlongto{\CC}
\bigoplus_{\ell=1}^{J-1} H(\ell,J)
    \longrightarrow\bigoplus_{\ell=1}^{J-1}
    \frac{H(\ell,J)}{H(\ell,J-1)},
\]
where the second map is the natural projection.  
\end{defn}

(Removing the ``top layer'' $\pH\to H(J,J)$ from the map in part (1)
of Proposition~\ref{prop:containments} in the definition of $\Phi_J$
is not strictly necessary, but it streamlines the analysis of a
certain determinant in the last part of the proof of
Theorem~\ref{thm:ref2}.)

  We refine Theorem~\ref{thm:ref1} as follows.

  \begin{thm}\label{thm:ref2}
    There is a non-empty Zariski open subset $U\subseteq A$ such that if
    $a\in U(k)$ and $Y=Y_a$ then:
    \begin{enumerate}
    \item     If $1< J\le (p+1)/2$, then $\Phi_J$ is surjective.
     \item If $(p+1)/2\le J<p$, then $\Phi_J$ is injective.
    \end{enumerate}
  \end{thm}

  We will prove Theorem~\ref{thm:ref2} in the next section.

  \begin{proof}[Proof that Theorem~\ref{thm:ref2} implies
    Theorem~\ref{thm:ref1}]
    Assume Theorem~\ref{thm:ref2} and fix $a\in U(k)$.  We first prove by
    induction that for $J\le(p+1)/2$, 
    \begin{equation}\label{eq:image}
    \CC\left(H^{\le J}\right)=\bigoplus_{\ell=1}^{J} H(\ell,J).      
    \end{equation}
    The base case $J=1$ follows from (2) and (3) of
    Proposition~\ref{prop:containments}.  Now assume $1<J\le(p+1)/2$ and that
    \eqref{eq:image} holds for $J-1$.  By
    Proposition~\ref{prop:containments}, part (3), $\pH$ maps onto $H(J,J)$,
    and by Theorem~\ref{thm:ref2}, part (1), ${}'H^{J}$ maps onto
    \[\bigoplus_{\ell=1}^{J-1} \frac{H(\ell,J)}{H(\ell,J-1)}.\]
  By the induction hypothesis,
\[  \CC\left(H^{\le J-1}\right)=\bigoplus_{\ell=1}^{J-1} H(\ell,J-1),\]
  so we conclude that \eqref{eq:image} holds for $J$ as well.  This
  proves  part (1) of Theorem~\ref{thm:ref1}.
  
  To prove part (2) of Theorem~\ref{thm:ref1}, note that if
  $J\ge(p+1)/2$, 
  Proposition~\ref{prop:containments} and part (2) of
  Theorem~\ref{thm:ref2} together show that $\CC$ defines an injection
  \begin{equation}\label{eq:inj}
  H^{J}\into\bigoplus_{\ell=1}^{J}
    \frac{H(\ell,J)}{H(\ell,J-1)}.
  \end{equation}
  (In the $\ell=J$ summand,
  the denominator $H(J,J-1)$ is zero.)  Now assume that $\omega\in H$
  has $\CC(\omega)=0$.  Write $\omega=\omega_{p-1}+\omega'$ where
    \[\omega_{p-1}\in H^{p-1}\and \omega'\in H^{\le p-2}.\]
    By Proposition~\ref{prop:containments}, $\CC(\omega')$ vanishes in
    the right hand side of \eqref{eq:inj} for $J=p-1$ and therefore
    so does $\CC(\omega_{p-1})$.  It then follows from the injectivity
    of \eqref{eq:inj} for $J=p-1$
    that $\omega_{p-1}=0$, i.e., that
    $\omega\in H^{\le p-2}$.  Repeating the argument using
    \eqref{eq:inj} for $J=p-2, p-3,\dots,(p+1)/2$
    shows that $\omega\in H^{(p-1)/2}$, and this is exactly the claim
    in part (2) of Theorem~\ref{thm:ref1}.  This completes the
    reduction of Theorem~\ref{thm:ref1} to Theorem~\ref{thm:ref2}.
  \end{proof}

\section{Proof of Theorem~\ref{thm:ref2}}
To prove Theorem~\ref{thm:ref2}, we will write down the matrix of
$\Phi_J$ in suitable coordinates and argue that this matrix has full
rank for generic choices of $a\in A(k)$.  More precisely, for each
value of $J$, we will find a non-empty Zariski open subset
$U_J\subseteq A$ on which the assertion of Theorem~\ref{thm:ref2} for
$J$ holds.  Since $A$ is irreducible, the intersection of the $U_J$ is
a non-empty open subset $U$ of $A$ on which all of
Theorem~\ref{thm:ref2} holds.

First we write down sets
indexing bases of the domain and range of
$\Phi_J$.

\begin{defns}\label{def:C+R}
For $1<J<p$, set  
\[C_J:=\left\{i'\left| 0<i'<\frac{(p-J)d}{p}, p\nodiv
      i'\right.\right\},\]
and
\[R_J=\left\{(i,\lambda)\left|\, 0<\lambda<J\text{ \ and \
      }\frac{(p\lambda-J+1)d}{p^2}
      <i<\frac{((p(\lambda+1)-J))d}{p^2}\right.\right\}.\]
\end{defns}

We order $C_J$ by the archimedean size of $i'$ and $R_J$ by the
archimedean size of $i$.  Note that the intervals in the definition of
$R_J$ for varying $\lambda$ do not overlap, so if
$(i,\lambda)\in R_J$, then $\lambda$ is uniquely determined by $i$.
Thus we sometimes write ``for $i\in R_J$''.

Clearly, $\{\omega_{i',J}|i'\in C_J\}$ is a basis of $\prH$.  Using
the definitions of $\Phi_J$ (Definition~\ref{def:Phi}) and $H(\ell,J)$
(Definition~\ref{def:H(l,e)}) and setting $\lambda=J-\ell$ in the
definition of $R_J$, we see that
$\{\omega_{i,J-\lambda}|(i,\lambda)\in R_J\}$ is a basis for the range
of $\Phi_J$.

Next we compare the sizes of the index sets $C_J$ and $R_J$.

\begin{prop}\label{prop:comp}\mbox{}
  \begin{enumerate}
  \item If $1< J\le(p+1)/2$, $|R_J|\le |C_J|$.
  \item If $(p+1)/2\le J<p$, $|C_J|\le |R_J|$.
  \end{enumerate}
\end{prop}

\begin{proof}
  (1), Supposing $1<J\le(p+1)/2$, we write down an injective map of sets
  $R_J\into C_J$. If $(i,\lambda)\in R_J$, then from the definition of  $R_J$,
  \[i<\frac{((p(\lambda+1)-J))d}{p^2}\le\frac{(pJ-J)d}{p^2}
    \le\frac{(p^2-1)d}{2p^2}<d/2.\]
  Rewriting the inequalities defining $R_J$, we have
  \[-\frac{(J-1)d}p<pi-\lambda d<\frac{(p-J)d}p,\]
  and since $J\le(p+1)/2$, it follows that
  the lower bound is $\ge -(p-J)d/p$.  Note that $0<\lambda<p$, so
  $\lambda$ is prime to $p$.  Let $i'=\pm(pi-\lambda d)$; this is
  non-zero, and we choose the sign so that $i'>0$.  Then we have that
  $0<i'<(p-J)d/p$ and $p\nodiv i'$,
  so $i'\in C_J$.  This defines a map
  $R_J\to C_J$.  If both $i_1$ and $i_2$ map to $i'$, then
  $i_1\equiv\pm i_2\pmod d$.  But $i_1$ and $i_2$ are $<d/2$, so this
  is possible only if $i_1=i_2$.  Thus the map $R_J\to C_J$ is
  injective.

  (2) Supposing $(p+1)/2\le J<p$, we write down an injective map of
  sets $C_J\into R_J$.  Choose $i'\in C_J$ and note that $i'<d/2$.
  Let $\iota$ be the unique integer in $(0,d)$ with
  $p\iota\equiv i'\pmod d$, and let $\rho=(p\iota-i')/d$.  By the
  definition of $C_J$, $p\nodiv i'$, so we have $0<\rho$, and since
  $\iota<d$, we have $\rho<p$.  If $\rho<J$ we claim that
  $(i=\iota,\lambda=\rho)$ is an element of $R_J$.  Indeed, we have
  \[ 0<pi-\lambda d=i'<\frac{(p-J)d}{p},\]
  and $0<\lambda<J$ as required.  If $\rho\ge J$, we claim that
  $(i=d-\iota,\lambda=p-\rho)$ is an element of $R_J$.  Indeed,
  \[-\frac{(J-1)d}{p}\le-\frac{(p-J)d}{p}<-i'=pi-\lambda d<0,\] and
  $0<\lambda<J$ as required.  Thus we have defined a map $C_J\to R_J$.
  If $i'_1$ and $i'_2$ both map to $i$, then
  $i_1'\equiv\pm i_2'\pmod d$, and since $i_1'$ and $i_2'$ are $<d/2$,
  this is only possible if $i'_1=i'_2$.  Thus the map $C_J\to R_J$ is
  injective. This completes the proof of the Proposition.
\end{proof}

Now take $a\in A(k)$, let $f=f_a$, and let $b_{\lambda,n}$ be the
coefficient of $x^n$ in $f^{\lambda}$.  We recall the calculation of
$\CC(\omega_{i,J})$ in the proof of
Proposition~\ref{prop:containments}: for $1<J<p$ and $i'\in C_J$,
we have
\begin{align*}
\CC(\omega_{i',J})
  &=\sum_{\ell=1}^J\binom{J-1}{\ell-1}(-1)^{J-\ell}
    y^{\ell-1}\CC\left(f^{J-\ell}x^{i'-1}\,dx\right)\\
  &=\sum_{\lambda=0}^{J-1}\binom{J-1}{\lambda}(-1)^\lambda
    y^{J-\lambda-1}\CC\left(f^\lambda x^{i'-1}\,dx\right),
\end{align*}
where the second expression is the result of substituting
$\ell=J-\lambda$.  Writing this in terms of the basis elements in
$R_J$, we see that the coefficient of $\omega_{i,J-\lambda}$ in
$\CC(\omega_{i',J})$ is
\[\binom{J-1}{\lambda}(-1)^\lambda b_{\lambda,pi-i'}^{1/p}.\] 
Thus to prove Theorem~\ref{thm:ref2}, we need to show that for each
$1<J<p$, the matrix 
\begin{equation}\label{eq:firstM}
  \left(\binom{J-1}{\lambda}(-1)^\lambda
  b_{\lambda,pi-i'}^{1/p}\right)_{\substack{(i,\lambda)\in R_J\\ i'\in C_J}}  
\end{equation}
has maximal rank.

Now $J$ is fixed, so the binomial coefficient and the sign of the
entry at $((i,\lambda),i')$ depend only on the row index
$(i,\lambda)$.  We may thus divide each element of row $(i,\lambda)$
by the non-zero quantity
\[\binom{J-1}{\lambda}(-1)^\lambda\]
without changing the rank of the matrix.  We may also raise every
element in the matrix to the $p$-th power without changing the
rank.
%(This holds because, for example, the minors of the transformed
%matrix vanish precisely when the minors of the original matrix do.)
Thus, to prove Theorem~\ref{thm:ref2}, we must show that for generic
$a$, the matrix
\begin{equation}\label{eq:M-special}
\Bigl(b_{\lambda,pi-i'}\Bigr)_{\substack{(i,\lambda)\in R_J\\ i'\in C_J}}  
\end{equation}
has maximal rank.

We change point of view and make calculations in the polynomial ring
\[S=k[a_0,\dots,a_d].\]
(We have delayed doing so until now due to the exponent $1/p$ in
equation~\eqref{eq:firstM}.)  Let $f$ be the generic polynomial of
degree $d$:
\[f=a_0+a_1x+\cdots+a_dx^d\in S[x],\]
and for $\lambda<p$ and $n\ge0$, let $b_{\lambda,n}$ be the coefficient of
$x^n$ in $f^\lambda$, so
\[b_{\lambda,n}=\sum_{\substack{e_0,\dots,e_d\ge0\\
      e_0+e_1\cdots+e_d=\lambda
      \\0e_0+1e_1+\cdots+de_d=n}}
  \left(\frac{\lambda!}{e_0!e_1!\cdots e_d!}\right)a_0^{e_0}\cdots a_d^{e_d}.\]
Note that this is zero if and only if $n>\lambda d$.

\begin{defn}\label{def:M-generic}
Fix a $J$ with $1<J<p$ and define a matrix $M=M_J$ with rows indexed
by $R_J$, columns indexed by $C_J$, and with entries in $S$ by the
assignment
\begin{equation*}
M_{i,i'}=M_{(i,\lambda),i'}=b_{\lambda,pi-i'},
\end{equation*}
where $b_{\lambda,n}$ is the coefficient of $x^n$ in $f^\lambda$.
\end{defn}

Thus the entries of $M$ are elements of $S$, and if we assign values
in $k$ to the variables $a_0,\dots,a_m$, $M$
specializes to the matrix in equation~\eqref{eq:M-special}.  Our task
is to show that there is a non-empty Zariski open set $U_J\subseteq A$
such that for $a\in U_J(k)$, the specialization has maximal rank.

Here is the example $p=5$, $d=18$, $J=3$:
\[R_J=\{(3,1),(4,1),(5,1),(6,2),(7,2),(8,2)\},
  \qquad
  C_J=\{1,2,3,4,6,7\},\]
and
\[M_J=\begin{pmatrix}
b_{1,14}&b_{1,13}&b_{1,12}&b_{1,11}&b_{1,9}&b_{1,8}\\
0&b_{1,18}&b_{1,17}&b_{1,16}&b_{1,14}&b_{1,13}\\
0&0&0&0&0&b_{1,18}\\
b_{2,29}&b_{2,28}&b_{2,27}&b_{2,26}&b_{2,24}&b_{2,23}\\
b_{2,34}&b_{2,33}&b_{2,32}&b_{2,31}&b_{2,29}&b_{2,28}\\
0&0&0&b_{2,36}&b_{2,34}&b_{2,33}
\end{pmatrix}.
\]

To finish the proof of Theorem ~\ref{thm:ref2}, we will show that
the determinant of a suitable maximal minor of $M$ is a non-zero
element of $S$.  Then the complement of the zero locus of this
determinant is the non-empty open $U_J$ we seek.

To that end, we now establish certain properties of the matrix
$M$.

\begin{prop}\label{prop:M-props}\mbox{}
  \begin{enumerate}
  \item If $M_{i,i'}=0$ \textup{(}i.e., if $pi-i'>\lambda d$\textup{)}
    then there is an $i''\in C_J$ such that $i''>i'$ and
    $M_{i,i''}=b_{\lambda,\lambda d}=a_d^\lambda$.
    In other words, the zero entries in $M$ all occur to the left of and
    in the same row as an occurrence of $b_{\lambda,\lambda d}$.  In
    particular, no row of $M$ is zero.
  \item  For all $((i,\lambda),i')\in R_J\times C_J$, we have $pi-i'>(\lambda-1)d$.
  \end{enumerate}
\end{prop}

\begin{proof}
(1) Suppose that $pi-i'>\lambda d$.
Setting $i''=pi-\lambda d$, we see that $pi-i''=\lambda d$ and 
$i''$ is $>i'$.  Also, $i''$ is prime to $p$.  Since $(i,\lambda)\in
R_J$, we have
  \[pi<\frac{(p(\lambda+1)-J)d}{p}\]
  so
  \[i''=pi-\lambda d<\frac{(p-J)d}{p},\]
  which ensures that $i''\in C_J$.

(2) It suffices to check this for the maximum value of
$i'$, namely $i'=\lfloor(p-J)d/p\rfloor$.  But $(i,\lambda)\in R_J$ implies
\begin{align*}
pi>\frac{(\lambda p-J+1)d}{p}&>(\lambda-1)d+\frac{(p-J+1)d}{p}\\
    &>(\lambda-1)d+\left\lfloor\frac{(p-J)d}{p}\right\rfloor,    
\end{align*}
which proves the claim.
\end{proof}

Let $T\subseteq R_J\times C_J$ be the subset where $pi-i'\le\lambda d$,
i.e., the set of $(i,i')$ where $M_{i,i'}\neq0$.  Define a function
$\phi:T\to\{1,\dots,d\}$ by $\phi(i,i')=pi-i'-(\lambda-1)d$.  We note
several simple but important properties of $\phi$.

\begin{prop}\label{prop:phi}
  We have:
  \begin{enumerate}
  \item $\phi(i,i')\equiv pi-i'\pmod d$
  \item The values of $\phi$ for a fixed $i$ and varying $i'$ are
    distinct.
  \item The values of $\phi$ for a fixed $i'$ and varying $i$ are also
    distinct.
  \end{enumerate}
\end{prop}

\begin{proof}\mbox{}
  (1) and (2) are obvious from the definition of $\phi$.

  (3) Since $0<i<d$, the values of $i$ are distinct modulo $d$, and
  since $p\nodiv d$, so are the values of $pi-i'$ for $i'$ fixed.  But
  $\phi(i,i')\equiv pi-i'\pmod d$, so the values of $\phi$ for a fixed
  $i'$ and varying $i$ are distinct.
\end{proof}

Here are the values of $\phi$ in the example where $p=5$, $d=18$, and
$J=3$:
\[\begin{pmatrix}
14&13&12&11&9&8\\
&18&17&16&14&13\\
&&&&&18\\
11&10&9&8&6&5\\
16&15&14&13&11&10\\
&&&18&16&15
\end{pmatrix}
\]

We now consider $N$, a maximal square minor of $M$.  If
$J\le(p+1)/2$, %\bry{comma?}
we let $N$ be the rightmost $|R_J|$ columns of $M$.  If $J\ge(p+1)/2$,
we let $N$ be the top $|C_J|$ rows of $M$.  In either case,
Proposition~\ref{prop:comp} shows that $N$ is square, so we may
consider its determinant.  We finish the proof of
Theorem~\ref{thm:ref2} by showing that this determinant is a non-zero
element of $S$.

Let $R_N\subseteq R_J$ be the indices of the rows of $N$ and let $C_N\subseteq
C_J$ be the indices of the columns of $N$.  Let $T_N\subseteq T$ be
defined by
\[T_N=T\bigcap\left(R_N\times C_N\right),\]
so $(i,i')\in T_N$ if and only if the entry $N_{i,i'}$ is non-zero.

We consider the permutation expansion of the determinant of $N$.
More precisely, for each bijection $\sigma:R_N\to C_N$, we form the
product
\[Pr(\sigma):=\prod_{i\in R_N} N_{i,\sigma(i)}\]
so that $\det(N)$ is the sum over all $\sigma$ of $\pm Pr(\sigma)$.

We now introduce a monomial ordering and prove that there is a unique
permutation $\sigma$ such that the leading monomial of $Pr(\sigma)$ is
maximal.  This monomial cannot by cancelled by any other monomial in
the permutation expansion of the determinant,
which will establish that the determinant is non-zero.

Define a monomial order on $S$
by declaring that
\[a_0^{e_0}\cdots a_d^{e_d}<a_0^{f_0}\cdots a_d^{f_d}\]
if for the largest index $i$ with $e_i\neq f_i$ we have $e_i<f_i$.
Thus $a_d$ is larger than any monomial in $k[a_0,\dots,a_{d-1}]$.
(For a brief review of the basic facts about monomial orderings, we
refer to \cite[Ch.~2, \S2]{CoxLittleOsheaIVA}.)

Define a grading on $S$ by declaring that $a_i$ has weight $i$.  Then
$b_{\lambda,n}$ is homogeneous of degree $\lambda$ and has weight $n$.

For $0\neq g\in S$,
we write $LM(g)$ for the leading monomial of $g$
with respect to the order above.

Recall that every $(i,\lambda)\in R_N$ is uniquely determined by $i$.
To ease notation, for the rest of the proof we always denote elements
of $R_N$ by $i$ (rather than by $(i,\lambda)$).  In particular, a pair
$(i,i')$ will always denote an element of $R_N\times C_N$.

For a bijection $\sigma:R_N\to C_N$, define its graph by
\[\Gamma(\sigma)=\left\{(i,\sigma(i))
    \mid i\in R_N\right\} \subseteq R_N\times C_N.\]

The following result establishes that $\det(N)\in S$ is non-zero, and
thus completes the proof of Theorem~\ref{thm:ref2}.

\begin{thm}\label{thm:det}\mbox{}
  \begin{enumerate}
  \item If $(i,i')\in T_N$, then
\[LM(N_{i,i'})=LM(b_{\lambda,pi-i'})=a_{\phi(i,i')}a_d^{\lambda-1}.\]
\item $Pr(\sigma)\ne0$ if and only if $\Gamma(\sigma)\subseteq T_N$.
\item If $\Gamma(\sigma)\subseteq T_N$, then
  \[LM(Pr(\sigma))=a_0^{e_0}a_1^{e_1}\cdots
    a_{d-1}^{e_{d-1}}a_d^{e_d+\Lambda}\]
  where for $1\le\ell\le d$,
  \[e_\ell=\left|\left\{i\in R_N|\phi(i,\sigma(i))=\ell\right\}\right|,\]
  and
  \[\Lambda=\sum_{i\in R_N}(\lambda-1).\]
  \textup{(}Recall that each index $i$ in $R_N$ determines a unique
  $\lambda$.\textup{)}
\item There exists a bijection $\sigma_0:R_N\to C_N$ such that
  $\Gamma(\sigma_0)\subseteq T_N$ and such that if
  $\sigma\neq\sigma_0$ is another bijection $\sigma:R_N\to C_N$ with
  $\Gamma(\sigma)\subseteq T_N$, then
  $LM(Pr(\sigma_0))>LM(Pr(\sigma))$.
\item $\det(N)$ is a non-zero element of $S$.
  \end{enumerate}
\end{thm}

\begin{proof}\mbox{}
  (1) Note that $b_{\lambda,pi-i'}$ is homogeneous of degree $\lambda$
  and has weight $pi-i'=(\lambda-1)d+\phi(i,i')$.  Clearly
  $a_{\phi(i,i')}a_d^{\lambda-1}$ is the largest monomial of degree
  $\lambda$ and weight $(\lambda-1)d+\phi(i,i')$, and it appears in
  $f^\lambda$ with coefficient $\lambda$ if $\phi(i,i')<d$ and
  coefficient 1 if $\phi(i,i')=d$.

  (2) The product $Pr(\sigma)$ is non-zero if and only if each of the factors is
  non-zero if and only if $\Gamma(\sigma)\subseteq T_N$.

  (3) If $\Gamma(\sigma)\subseteq T_N$, then
  \[ LM(Pr(\sigma))=\prod_{i\in R} LM(N_{i,\sigma(i)}),\]
  and the result then follows from part (1).

  (4) We construct $\sigma_0$ with a greedy algorithm with steps
  indexed by $d,d-1,\dots,1$.  (See below for an example.)  At step
  $d$, consider
\[\Rs_d:=\left\{i\in R_N\mid\text{there is an $i'\in C_N$ with
      $\phi(i,i')=d$}\right\}.\]
For each $i\in \Rs_d$, part (2) of Proposition~\ref{prop:phi}
guarantees that there is a unique $i'$ such that $\phi(i,i')=d$, and
we set $\sigma_0(i)=i'$. This defines a map $\sigma_0:\Rs_d\to C_N$.
Part (3) of Proposition~\ref{prop:phi} guarantees that $\sigma_0$ is
injective.  Let $\Cs_d=\sigma_0(\Rs_d)$.  (If $\Rs_d$ happens to be
empty, we set $\Cs_d=\emptyset$ and proceed to the next step of the
algorithm.)

Note that Proposition~\ref{prop:M-props} (1) says that any zero
entries of $N$ occur in the same rows as entries where $\phi(i,i')=d$,
so if $i\in R_N\setminus \Rs_d$ and $i'\in C_N\setminus \Cs_d$, then
$N_{i,i'}\neq0$, i.e., $(i,i')\in T_N$ and $\phi$ is defined at
$(i,i')$.

In step $d-1$ of the algorithm, let
\[\Rs_{d-1}:=\left\{i\in R_N\setminus \Rs_d\mid\text{there is an $i'\in
      C_N\setminus \Cs_d$ with $\phi(i,i')=d-1$}\right\}.\]
Again, for each $i\in \Rs_{d-1}$, there is a unique
$C_N\setminus \Cs_d$ such that $\phi(i,i')=d-1$, and we set
$\sigma_0(i)=i'$.  (If $\Rs_{d-1}=\emptyset$, we set
$\Cs_{d-1}=\emptyset$ and proceed to the next step.)  This defines
$\sigma_0$ on $\Rs_d\cup \Rs_{d-1}$ and this map is injective.  Define
$\Cs_{d-1}:=\sigma_0(\Rs_{d-1})$.

We now continue, decreasing the sought after value of $\phi$ by one at
each step.  Since $\phi$ takes values in $\{1,\dots,d\}$ and is
defined at all
$(i,i')\in (R_N\setminus \Rs_d)\times (C_N\setminus \Cs_d)$, after $d$
steps, we have defined disjoint unions
\[R_N=\Rs_d\cup\cdots\cup\Rs_1
  \and
  C_N=\Cs_d\cup\cdots\cup\Cs_1,\]
  and a bijection $\sigma_0:R_N\to C_N$ such that
  $\sigma_0(\Rs_\ell)=\Cs_\ell$ and such that for $i\in \Rs_\ell$, we
  have $\phi(i,\sigma_0(i))=\ell$.

  Now suppose $\sigma\neq\sigma_0$ is another bijection $R_N\to C_N$.
  Let $\ell$ be the largest integer in $\{1,\dots,d\}$ such that
  $\sigma |_{\Rs_\ell}\neq\sigma_{0}|_{\Rs_\ell}$. Using part (3) of
  Theorem~\ref{thm:det}, we see that the exponents of
  $a_{\ell+1},\dots,a_d$ in $LM(Pr(\sigma))$ and $LM(Pr(\sigma_0))$
  are equal, but the exponent of $a_\ell$ in $LM(Pr(\sigma_0))$ is
  strictly larger than the exponent of $a_\ell$ in $LM(Pr(\sigma))$.
  This shows that $LM(Pr(\sigma_0))>LM(Pr(\sigma))$, as desired.

  (5) Part (4) shows that $LM(Pr(\sigma_0))$ is strictly larger than
  $LM(Pr(\sigma))$ for any bijection $\sigma\neq\sigma_0$ such that
  $Pr(\sigma)\neq0$.  Thus $LM(Pr(\sigma_0))$ cannot be cancelled by
  any monomial in $Pr(\sigma)$ for any $\sigma\neq\sigma_0$, and we
  conclude that $LM(Pr(\sigma_0))$ appears in $\det(N)$ with a
  non-zero coefficient.  This shows that $\det(N)\neq0$ and completes
  the proof of the Theorem.
\end{proof}

In the example where $p=5$, $d=18$, and $J=3$, $R_N=\{3,4,5,6,7,8\}$,
$C_N=\{1,2,3,4,6,7\}$, and the non-zero subsets $\Rs_\ell$ and
$\Cs_\ell$ are
\begin{center}
\begin{tabular}{c c}
$\Rs_{18}=\{4,5,8\}$,&$\Cs_{18}=\{2,4,7\}$,\\
$\Rs_{16}=\{7\}$,&$\Cs_{16}=\{1\}$,\\
$\Rs_{12}=\{3\}$, &$\Cs_{12}=\{3\}$,\\
$\Rs_6=\{6\}$,&$\Cs_6=\{6\}$.  
\end{tabular}
\end{center}
The permutation $\sigma_0$ takes the following values:
\[    \begin{array}{c|cccccc}
      i & 4 & 5 & 8& 7 & 3 & 6\\
      \hline
      \sigma_0(i)&2& 7 & 4 & 1 & 3 & 6
    \end{array}\]
  and
  \[LM(\Pr(\sigma_0))=a_6a_{12}a_{16}a_{18}^6,\]
  \begin{proof}[Proof that Theorem~\ref{thm:det} implies
    Theorem~\ref{thm:ref2}]
    For each $J$, let $U_J\subseteq A$ be the complement of the zero set
    of the determinant of $N$ studied in Theorem~\ref{thm:det}.  By
    that result, $U_J$ is non-empty.  If $U=\cap_J U_J$ and if
    $a\in U(k)$, then for each $J$ in $\{2,\dots,p-1\}$, the determinant
    of $N$ specializes to a non-zero element of $k$, so $N$ has full
    rank.  This implies that the matrix $M_J$ and the map $\Phi_J$
    computed for $Y_a$ have maximal rank, and this is exactly the
    assertion of Theorem~\ref{thm:ref2}.
  \end{proof}

  \begin{rem}
    It is clear that the open set $U$ is defined over the prime field,
    but we do not know whether it always has a rational point over
    that field.  Thus we may need to extend $k$ to guarantee the
    existence of a point $a\in U(k)$.
  \end{rem}

\bibliography{database}

\end{document}

%% file: formatting.tex
\numberwithin{equation}{section}

\theoremstyle{plain}
\newtheorem{thm}[subsection]{Theorem}

\newtheorem{prop}[subsection]{Proposition}

\newtheorem{cor}[subsection]{Corollary}

\newtheorem*{thm*}{Theorem}
\newtheorem*{mthm*}{Main Theorem}
\newtheorem{prop-def}[subsection]{Proposition-Definition}
\newtheorem{def-prop}[subsection]{Definition-Proposition}
\newtheorem{def-lemma}[subsection]{Definition-Lemma}

\theoremstyle{definition}
\newtheorem{defn}[subsection]{Definition}
\newtheorem{defns}[subsection]{Definitions}
\newtheorem*{defn*}{Definition}

\theoremstyle{remark}
\newtheorem{rem}[subsection]{Remark}
\newtheorem{rems}[subsection]{Remarks}

\newtheorem{s-example}[subsection]{Example}
\newtheorem{s-examples}[subsection]{Examples}

%% file: macros.tex
\usepackage[OT2,T1]{fontenc}
\DeclareSymbolFont{cyrletters}{OT2}{wncyr}{m}{n}
\DeclareMathSymbol{\sha}{\mathalpha}{cyrletters}{"58}

\newcommand{\CC}{\mathcal{C}}

\newcommand{\kbar}{{\overline{k}}}

\newcommand{\Z}{\mathbb{Z}}

\newcommand{\A}{\mathbb{A}}
\renewcommand{\P}{\mathbb{P}}

\newcommand{\<}{\langle}
\renewcommand{\>}{\rangle}
\newcommand{\into}{\hookrightarrow}

\newcommand{\isoto}{\,\tilde{\to}\,}

\newcommand{\nodiv}{\not|}

\def\nodiv{\mathrel{\mathchoice{\not|}{\not|}{\kern-.2em\not\kern.2em|}
{\kern-.2em\not\kern.2em|}}}

\newcommand{\labeledto}[1]{\overset{#1}{\to}}
\newcommand{\labeledlongto}[1]{\overset{#1}{\longrightarrow}}

\DeclareMathOperator{\ord}{ord}

\DeclareMathOperator{\gal}{Gal}

\def\clap#1{\hbox to 0pt{\hss#1\hss}}

\makeatletter
\newcommand*\bigcdot{\mathpalette\bigcdot@{.5}}
\newcommand*\bigcdot@[2]{\mathbin{\vcenter{
               \hbox{\scalebox{#2}{$\m@th#1\bullet$}}}}}
\makeatother